\documentclass[10pt]{amsart}
\usepackage[cp1251]{inputenc}
\usepackage[english]{babel}
\usepackage{amsmath}
\usepackage{amssymb}
\usepackage{amsfonts}
\usepackage{srcltx} 
\usepackage[dvips]{graphicx}

\newtheorem{lemma}{Lemma}
\newtheorem{theorem}{Theorem}

\newtheorem{corollary}{Corollary}

\newcommand{\F}{{\mathbb{F}}}

\begin{document}

\title{On components of a Kerdock code and
 the dual of the BCH code $C_{1,3}$}
\author{{I. Yu. Mogilnykh,  F. I. Solov'eva}}%

\address{Ivan Yu. Mogilnykh
\newline\hphantom{iii} Sobolev Institute of Mathematics,
\newline\hphantom{iii} pr. ac. Koptyuga 4,
\newline\hphantom{iii} 630090, Novosibirsk, Russia}%
\email{ivmog84@gmail.com}%

\address{Faina I. Solov'eva
\newline\hphantom{iii} Sobolev Institute of Mathematics,
\newline\hphantom{iii} pr. ac. Koptyuga 4,
\newline\hphantom{iii} 630090, Novosibirsk, Russia}%
\email{sol@math.nsc.ru}%

\thanks{\copyright \ 2018  I. Yu. Mogilnykh,  F. I. Solov'eva}
\thanks{\rm This work was funded by the Russian Science Foundation under grant
18-11-00136.
}



\maketitle

\begin{quote}
{\small \noindent{\sc Abstract. } In the paper we investigate the
 structure of $i$-components of two classes of codes: Kerdock
 codes and the duals of the primitive cyclic BCH code with designed
distance 5 of length $n=2^m-1$, for odd $m$. We prove that for any
admissible length a punctured Kerdock code consists of two
$i$-components and the dual of BCH code is $i$-component for any
$i$. We give an alternative proof for the fact that the
restriction of the Hamming scheme to a doubly shortened Kerdock
code is an association scheme \cite{vanCaen}.

\medskip

}
\end{quote}

\section{Introduction}

Let $\F^n$ be the vector space of dimension $n$ over the Galois
field $GF(2)$.
Denote by ${\bf 0}^n$ and ${\bf 1}^n$  the all-zero and all-one
vectors in $\F^n$ respectively. The Hamming distance $d(x,y)$
between vectors $x,y \in \F^n$ is the number of positions at which
the corresponding symbols in $x$ and $y$ are different. {\it The
Hamming weight} $w(x)$ of a vector $x$ is $d(x,{\bf 0}^n)$. A {\it
code} of length $n$ is a subset of $\F^n$. Vectors of a code are
called {\it codewords}. The {\it size} of a code is the number of
its codewords. The {\it code distance} (or {\it minimum distance})
of a code is the minimum value of the Hamming distance between two
different codewords from the code. The {\it kernel} $Ker(C)$ of a
code $C$ is $\{x:x+C=C\}$. Obviously, the code $C$ is a union of
cosets of $Ker(C)$. The code obtained from a code $C$ by deleting
one coordinate position is called  the {\it punctured code}. Such
code we denote by $C^*$ and doubly punctured code by $C^{**}$. The
{\it shortened code} of $C$ is obtained by selecting the subcode
of $C$ having zeros at a certain position and deleting this
position. We denote such code by $C^{\prime}$. Doubly shortened
code we denote by $C^{\prime\prime}$. For a code $C$ denote by
$I(C)$ the set of distances between its codewords:
$I(C)=\{d(x,y):x,y \in C\}$ and by $C_i$ denote the set of its
codewords of weight $i$: $C_i=\{x\in C: w(x)=i\}$. All other
necessary definitions and notions can be found in \cite{MWSl}.

Given a code $C$ with minimum distance $d$ consider the graph
$G_i(C)$ with the set of codewords as the set of vertices and the
set of edges $\{(x,y):d(x,y)=d, x_i\neq y_i\}$.  A connected
component of the graph $G_i(C)$ is called the {\it $i$-component}
of the code. If the minimum distance $d$ is greater then $2$ then
changing the value in $i$th coordinate position in all vectors of
any $i$-component by the opposite one in the code leads to a code
with the same parameters: length, size and code distance.
Therefore, we can obtain an exponential number (as a function of
the number of $i$-components in the code) of different codes with
the same parameters. Such approach was earlier successfully
developed for the class of perfect codes. The method of
$i$-components allowed to construct a large class of pairwise
nonequivalent perfect codes and was used to study various code
properties, see the survey \cite{Sol}.

  Punctured Preparata codes, perfect codes with code distance 3 and the primitive cyclic BCH code $C_{1,3}$ with designed distance 5
  of length $2^m-1$, odd $m$ are known to be uniformly packed
 \cite{SZZ1971}, \cite{BZZ}. Therefore, the fixed weight codewords of the extensions of these codes form 3-designs,
which was proved by Semakov, Zinoviev and Zaitsev in
\cite{SZZ1971}. An analogous property holds for duals of codes
from these classes. Let $C^{\perp}$ be a formally dual code to a
code $C$ with code distance $d$, i.e. their weight distributions
are related by McWilliams identities \cite{MWSl}. In Theorem 9,
Ch. 9, \cite{MWSl} it was shown that the set of codewords of any
fixed weight in $C^{\perp}$ is $(d-\bar{s})$-design, where
$\bar{s}$ denotes the number of different nontrivial (not equal to
$0$ and $n$) weights of the codewords of $C^{\perp}$. It is
well-known
 that a Kerdock code and a Preparata code of the same length are
 formally dual. Therefore, the fixed weight codewords of a Kerdock code are  $3$-designs
 and the code $C^{\perp}_{1,3}$ orthogonal to $C_{1,3}$ of length
 $2^m-1$, $m$-odd,  are $2$-designs
respectively.

The aforementioned codes are related to association schemes. Let
$X$ be a set, and there are $n+1$ relations $R_i$, $i\in I$ that
partition $X\times X$. The pair $(X,\{R_i\}_{i \in I})$ is called
an {\it association scheme}, if there are $\delta_{i,j}^k(X)$,
such that
\begin{itemize}
    \item The relation $\{(x,x):x \in X\}$ is $R_j$ for some $j \in I$.
    \item For any $i$, the relation $R_{i}^{-1}=\{(y,x):(x,y) \in
    R_i\}$ is $R_j$ for some $j\in I$.
    \item For any $i,j,k\in I$ and $x,y$ in $X$, $(x,y) \in R_i$ the following holds: $$\delta_{i,j}^k(X)=|\{z:z \in X, (x,z)\in R_j, (y,z)\in R_k
    \}|.$$
\end{itemize}
The numbers $\delta_{i,j}^k(X)$, $i,j,k\in I$ are called {\it
intersection numbers} of the association scheme.

Let $C$ be a binary code. Consider the partition of the cartesian
square $C\times C$ into distance relations, i.e. two pairs of
codewords are in the same relation if and only if the Hamming
distances between the pairs coincide. Such partition is called
{\it the restriction} of the Hamming scheme to the code $C$, see
\cite{Del}. There are several cases where the restriction gives an
association scheme. In this case, the code with this property is
called distance-regular, see \cite{SolTok}. Using linear
programming bound, Delsarte in \cite{Del} showed that the
restriction of the Hamming scheme to a shortened Kerdock code is
an association scheme. An analogous fact for Kerdock codes was
proved in \cite{SolTok} by finding the intersection numbers of the
restricted scheme directly. In work \cite{vanCaen}, see also
\cite{Abdukhalikov},  it is shown that the restriction to a doubly
shortened Kerdock code is also an association scheme. The latter
fact contributes to a significant part of the
 current paper concerning components of a Kerdock code, however
 we give an alternative combinatorial
proof for this fact as we essentially need a convenient way of
finding the intersection numbers of the scheme. Delsarte (Theorem
6.10, \cite{Del}) proved that the restriction of the Hamming
scheme to the dual of any linear uniformly packed code (in
particular, the code $C^{\perp}_{1,3}$, which is dual of the BCH
code $C_{1,3}$) is an association scheme.

In this paper we show that the punctured Kerdock code have two
$i$-components for any coordinate position $i$, while the dual of
a linear uniformly packed code with parameters of BCH code
$C_{1,3}$ is $i$-component for any coordinate position $i$.

\section{Components of Kerdock code}

In the section we fix $n$ to be $2^m$, for even $m$, $m\geq 4$.
 A {\it Kerdock code} $K$ is a binary code of length $n$,
 and minimum distance $d=(n-\sqrt{n})/2$,
consisting of the first order Reed--Muller code RM$(1,m)$ and
$2^{m-1}-1$ its cosets such that the weights of the codewords in a
coset are $d$ or $n-d$. These codes were firstly constructed in
\cite{Kerdock} and further generalizations were obtained in
\cite{Kantor}, \cite{Ham}.

The weight distribution of a Kerdock code is well-known and is
related with the weight distribution of a Preparata code via
McWilliams identities \cite{MWSl}.

\bigskip

\begin{tabular}{|c|c|}
  \hline
  i & The number of codewords of weight i \\
  \hline
  0 & 1 \\
  d & $n(n-2)/2$ \\
  $\frac{n}{2}$ & $2n-2$ \\
  n-d & $n(n-2)/2$ \\

    n & 1 \\
  \hline

\end{tabular}

\bigskip


In order to prove that a Kerdock code consists of two
$i$-component we use the following properties of the code, that
come from its definition. Without loss of generality, ${\bf 0}^n$
is in a Kerdock code.

(K1) Any  code $K$ is a union of $n/2$ cosets of RM$(1,m)$.

(K2) It is true that $K_{n/2}\bigcup \{ {\bf 0}^n, {\bf 1}^n \}=
\mbox{ RM}(1,m)$.

(K3) The distance between codewords from different cosets of
RM$(1,m)$ in the code $K$ is either $d$ or $n-d$.

(K4) Nonzero distances between codewords in any coset are either
$n/2$ or $n$.

(K5) RM$(1,m) \subseteq Ker(K)$.

The  property below follows from (K2)-(K5):

 (K6) If for $x,y
\in K$ we have $w(x+y)=n/2$ then $x+y\in K$.

\begin{theorem}\label{TMW}\cite{MWSl}[Theorem 9, Ch. 9]
Let $C$ be a code of length $n$ and minimum distance $d$,
$C^{\perp}$ be a code which is formally dual to $C$,
$\bar{s}=|I(C^{\perp})\setminus\{0, n\}|$. Then the set of
codewords of any fixed nonzero weight in $C^{\perp}$ is
$(d-\bar{s})$-design.
\end{theorem}

 Theorem \ref{TMW} applied to  Preparata and
Kerdock  codes implies the following:

(K7)\cite{MWSl} $K_d$, $K_{n/2}$, $K_{n-d}$ are 3-designs.

\medskip


In order to proceed further we need the following lemma.

\begin{lemma}\label{magic}
Let $x$ be a vector of weight $i$, $D$ be
$1-(n,j,\lambda_1)$-design. Let the distance between $x$ and
vectors of $D$ take values $k_1,\ldots, k_s$ with multiplicities
$\delta^{k_1},\ldots,\delta^{k_s}$ respectively. Then the
following formula holds: \begin{equation}\label{lemma1*} \sum_{l=
1}^{s}\delta^{k_l}\cdot \frac{i+j-k_l}{2}=i\lambda_1
\end{equation} and $\delta^{k_1},\delta^{k_2}$ are uniquely
defined by $\delta^{k_3},\ldots,\delta^{k_s}$.
\end{lemma}

\begin{proof}
 Let the distance between the vector
$x$ and an arbitrary vector $y$ from $D$ be $k_l$, then there are
$$\frac{i+j-k_l}{2}$$ common unit coordinates for $x$ and $y$,
$l=1,2,\ldots,k_s$. On the other hand, there are exactly
$\lambda_1$ vectors of $D$ that have a prefixed coordinate to be
$1$. Double counting of $$\sum_{y\in D}|\{i: x_i=y_i=1\}|$$ gives
$\sum_{l= 1}^{s}\delta^{k_l}\cdot \frac{i+j-k_l}{2}=i\lambda_1 $.
Finally $\delta^{k_1},\delta^{k_2}$ are uniquely defined by
(\ref{lemma1*}) taking into account that $\sum_{l=
1}^{s}\delta^{k_l}=|D|$, where $|D|=\lambda_1 \frac{n}{j}.$
\end{proof}

Note that $I(K^{\prime\prime})=\{0,d,n/2,n-d\}$, as we exclude the
all-one vector in $K^{\prime}$.

\begin{theorem}\label{ShKass}
The restriction of the Hamming scheme to a doubly shortened
Kerdock code $K^{\prime\prime}$ is an association scheme.
\end{theorem}
\begin{proof}

In the proof of the current theorem we use the following
convention. By $\delta^k_{i,j}(x)$ we denote the number of
codewords of weight $j$ in $K^{\prime\prime}$ at distance $k$ from
the weight $i$ codeword $x$ in $K^{\prime\prime}$. Obviously, the
restriction of the Hamming scheme to $K^{\prime\prime}$ is an
association scheme if $\delta^k_{i,j}(x)$ for all $i, j, k \in
I(K^{\prime\prime})$ are shown to be independent on the choice of
a codeword $x$ of weight $i$ regardless of translation of
$K^{\prime\prime}$ by its codeword. The proof below relies only on
properties (K1)-(K7) of a Kerdock code $K$ that are independent on
the translation of the code.


\begin{lemma} \label{lemma2}
The number $\delta^k_{i,j}(x)$ does not depend on the choice of a
codeword $x$ in $K^{\prime\prime}_i$ if $i$ or $j$ equals to
$n/2$.
\end{lemma}

\begin{proof}The property (K4) implies that the distances between codewords from
$K^{\prime\prime}_{n/2}$ and $K^{\prime\prime}_{d}$ or
$K^{\prime\prime}_{n-d}$ cannot be $n/2$. Moreover (K7) implies
that the sets of the fixed weight codewords of a doubly shortened
Kerdock code are 1-designs, so by Lemma \ref{magic}, the
intersection numbers $\delta_{i, j}^d(x)$ and $\delta_{i,
j}^{n-d}(x)$ are uniquely determined and do not depend on a choice
of $x$ if $i$ and $j$ are not equal to  $n/2$ simultaneously.

 Finally, $RM(1,m)^{\prime}$ is a linear Hadamard
code, so the set of nonzero codewords $K^{\prime\prime}_{n/2}$ of
its shortening $(RM(1,m)_{n/2})^{\prime\prime}$ are also at
distance $n/2$ apart pairwise, so $\delta_{n/2,n/2}^k(x)$ is
$n/4-1$ if and only if $k=n/2$ and is zero otherwise.
\end{proof}

\begin{lemma} \label{lemma3}
Let $n/2\in \{i,j,k\}$. Then the number $\delta^k_{i,j}(x)$ does
not depend on the choice of a codeword $x$ of weight $i$.
\end{lemma}

\begin{proof}

 We
show that $\delta_{i, j}^{n/2}(x)$=$\delta_{i, n/2}^{j}(x)$.
Consider the set $\{z\in K^{\prime\prime}_j, d(z,x)=n/2\}$. By
definition it is of the  size $\delta^{n/2}_{i,j}(x)$. Consider
the translation of the set by $x\in K^{\prime\prime}_{n/2}$. Since
$x$ is of weight $n/2$, the property (P6) implies that $x+z$ is a
codeword of the doubly shortened Kerdock code $K^{\prime\prime}$.
The substitution $z^{\prime}=z+x$ gives  the equality
$$\{z+x:z\in K^{\prime\prime}_j, d(z,x)=n/2\} =\{z^{\prime}\in
K_{n/2}^{\prime\prime}, d(z^{\prime},x)=j\}.$$ The cardinality of
the right hand side is $\delta^{j}_{i,n/2}(x)$, so
$\delta^{n/2}_{i,j}(x)=\delta^{j}_{i,n/2}(x)$ and the number is
independent on $x$ by Lemma \ref{lemma2}.
\end{proof}

\medskip

\begin{lemma} \label{lemma4}
The number $\delta^k_{i,j}(x)$ does not depend on the choice of a
codeword $x$ of weight $i$ for $i,j,k \in I(K^{\prime\prime})$.
\end{lemma}

\begin{proof} Since $I(K^{\prime\prime})=\{0,d,n/2,n-d\}$, the
nonzero distances between codewords from $K^{\prime\prime}_i$ and
$K^{\prime\prime}_j$ take not more than three nontrivial values.
 The property (K7)
implies that $K_j^{\prime\prime}$ is a 1-design and by Lemma
\ref{lemma3} the number $\delta_{i,j}^{n/2}(x)$ of codewords at
distance $n/2$ in $K_j^{\prime\prime}$ from  $x$ is independent on
choice of $x$ in $K_i^{\prime\prime}$, so the numbers
$\delta^d_{i,j}(x)$ and $\delta^{n-d}_{i,j}(x)$ are independent on
$x$ by Lemma \ref{magic}.
\end{proof}

The considerations in the beginning of the proof of the theorem
and Lemma \ref{lemma4} imply that the restriction of the Hamming
scheme to $K^{\prime\prime}$ is an association scheme.

\end{proof}

In order to find components of the punctured Kerdock code, we need
one more lemma.

\begin{lemma} \label{lemma5}
Let $C$ be a code of length $n'$  such that the restriction of the
Hamming scheme to its codewords is an association scheme. Let
$I(C)$ be such that $I(C) \cap \{n'-i: i \in I(C)\}=\varnothing$.
Then the restriction of the Hamming scheme to the code ${\overline
C}=C\bigcup({\bf 1}^{n'}+C)$ is an association scheme.
\end{lemma}
\begin{proof}
If $i$ is in $I(C)$, denote by $i'$ the number $n'-i$. If there
are given three distances from $I({\overline C})$ and even
belonging to $I(C)$ is even then the
 corresponding intersection number of ${\overline
C}$ is zero: $$\delta_{i',j}^k({\overline
C})=\delta_{i,j'}^k({\overline C})=\delta_{i,j}^{k'}({\overline
C})=\delta_{i',j'}^{k'}({\overline C})=0.$$
 Otherwise, the intersection number of ${\overline
C}$ coincides with that of $C$:
 \begin{equation}\label{eqal}\delta_{i',j'}^k({\overline
 C})=\delta_{i',j}^{k'}({\overline
 C})=\delta_{i,j'}^{k'}({\overline
 C})=\delta_{i,j}^k({\overline
 C})=\delta_{i,j}^k(
 C).\end{equation}
\end{proof}

\begin{theorem} \label{Comp}
Let $K^*$ be a punctured Kerdock code, $i \in \{1,\ldots, n-1\}$.
The code $K^*$ consists of two $i$-components and codewords are in
the same component if their puncturings in $i$th position have
weights of the same parity.
 \end{theorem}

\begin{proof} Consider any two coordinates $i, j$ of a Kerdock code of
length $n$. Proving that  there are just two $i$-components in
$K^{*}_j$ is equivalent to showing that the minimum distance graph
of the doubly punctured Kerdock code $K^{**}_{ij}$ has two
connected components (which are actually even and odd weight
codewords). Recall \cite{A} that the minimum distance graph of a
code is the graph with vertex set being codewords and edgeset
being pairs of codewords at code distance.

The minimum distance of the code $K^{**}$ is even and equal to
$d-2$. The even weight codewords of $K^{**}_{ij}$ are obtained
from codewords of $K$ having 0 or 1 simultaneously in $i$th and
$j$th positions by puncturing in these positions and the odd
weight codewords of $K^{**}_{ij}$ are obtained from the codewords
of $K$ having both 0 and 1 in $i$th or $j$th positions by
puncturing in these positions. Moreover, the odd weight subcode
$K^{**}_{ij}$ is obtained as a translation of even weight subcode
$K^{**}_{ij}$. Indeed, let $x$ be in $RM(1,m)$, having 0 in $i$th
position and 1 in $j$th position (there is such vector in the code
$RM(1,m)$ since codewords of $RM(1,m)$ of weight $n/2$ form
3-design). Since $x$ is in $Ker(K)$, the addition of even weight
codewords of $K^{**}_{ij}$ with the codeword $x^{**}$ obtained
from $x$ by puncturing in $i$th and $j$th position is the odd
weight subcode of $K^{**}_{ij}$.

In view of the  above, it is enough to show the connectedness of
the minimum distance graph of the even weight subcode of $K^{**}$,
whose codewords have weights from
$\{0,d-2,d,n/2-2,n/2,n-d-2,n-d,n-2\}$. The proof significatively
relies on the fact that the restriction of the Hamming scheme to
${\overline {K^{\prime\prime}}}$ is an association scheme which
follows from Theorem \ref{ShKass} and Lemma \ref{lemma5}. We show
that certain intersection numbers of the restriction of the
Hamming scheme to ${\overline {K^{\prime\prime}}}$  are nonzeros.

 \begin{lemma} The following equalities hold:

\begin{equation}\label{eqL1} \delta_{d-2,n/2}^{d-2}({\overline
{K^{\prime\prime}}})=\frac{n^2-6n-2nd+8d}{4(n-2d)}.\end{equation}
 \end{lemma}

 \begin{equation}\label{eqL2} \delta_{d-2,n/2}^{n-d-2}({\overline
{K^{\prime\prime}}})=\frac{n^2-2nd+2n}{4(n-2d)}.
 \end{equation}
\begin{proof} By equality (\ref{eqal}), we know that
$\delta_{n-d,n/2}^{n-k-2}(K^{\prime\prime})=\delta_{d-2,n/2}^{k}({\overline
{K^{\prime\prime}}})$ for $k=d,n-d$.
It is easy to see that the nonzero codewords of the code
$RM(1,m)^{\prime\prime}$ form $1-(n-2,n/2,n/4)$-design, since
there are exactly $2n-2$ nonzero codewords of $RM(1,m)$ of weight
$n/2$ which form 3-design.  From (K3) we have that
$\delta_{n-d,n/2}^{n-d}(K^{\prime\prime})+\delta_{n-d,n/2}^{d}(K^{\prime\prime})$
is the number of nonzero codewords of $RM(1,m)^{\prime\prime}$, so
it is $n/2-1$. Therefore, we obtain the following equality from
Lemma \ref{magic}:

$$\delta_{n-d,n/2}^{n-d}(K^{\prime\prime})\frac{n}{4}+(n/2-1-\delta_{n-d,n/2}^{n-d}(K^{\prime\prime}))(\frac{3n}{4}-d)=\frac{n}{4}(n-d).$$
and we find that

\medskip

\noindent
$\delta_{n-d,n/2}^{n-d}(K^{\prime\prime})=\frac{n^2-6n-2nd+8d}{4(n-2d)}$,
\,
$\delta_{n-d,n/2}^{d}(K^{\prime\prime})=\frac{n^2-2nd+2n}{4(n-2d)}.$
\end{proof}

From the values given by (\ref{eqL1}) and (\ref{eqL2}) we see that
$\delta_{d-2,n/2}^{d-2}({\overline {K^{\prime\prime}}})$ and
$\delta_{d-2,n/2}^{n-d-2}({\overline {K^{\prime\prime}}})$ are
nonzeros, which is equivalent to

\begin{equation}\label{eq2}
\delta_{d-2,n/2}^{d-2}({\overline {K^{\prime\prime}}}) \neq 0,
\,\,  \delta_{n/2,n-d-2}^{d-2}({\overline {K^{\prime\prime}}})\neq
0.\end{equation}

Consider the codewords of ${\overline {K^{\prime\prime}}}_{d-2}$.
Obviously, the codewords cannot be at distance $n/2$ pairwise
apart, which follows, for example, from the Plotkin bound.
Therefore there are codewords of weight $d-2$ at distance $d$
apart and $\delta_{d-2,d-2}^d({\overline {K^{\prime\prime}}})\neq
0$, which is equivalent to
\begin{equation}\label{eq1}
\delta_{d-2,d}^{d-2}({\overline {K^{\prime\prime}}})\neq 0.
\end{equation}

From (\ref{eq2}) we see that any codeword of ${\overline
{K^{\prime\prime}}}_{n/2}$ is at distance $d-2$ from at least one
codeword of $K_{d-2}$ and a codeword of ${\overline
{K^{\prime\prime}}}_{n-d-2}$ is at distance $d-2$ from at least
one codeword of ${\overline {K^{\prime\prime}}}_{n/2}$. Therefore,
${\overline {K^{\prime\prime}}}_{d-2}$, ${\overline
{K^{\prime\prime}}}_{n/2}$, ${\overline
{K^{\prime\prime}}}_{n-d-2}$ are in one connected component of the
minimum distance graph of ${\overline {K^{\prime\prime}}}$.
 Taking
into account the equality (\ref{eqal}) this fact is equivalent to
the fact that the codewords of ${\overline
{K^{\prime\prime}}}_{n-d}$, ${\overline
{K^{\prime\prime}}}_{n/2-2}$ and ${\overline
{K^{\prime\prime}}}_{d}$ belong to one component. Finally, the
inequality (\ref{eq1}) implies that ${\overline
{K^{\prime\prime}}}_{d-2}$ and ${\overline
{K^{\prime\prime}}}_{d}$ are in one component, which implies that
the codewords of weights $\{0,d-2,d,n/2-2,n/2,n-d-2,n-d,n-2\}$ are
in one connected component, which is exactly the minimum distance
graph of ${\overline {K^{\prime\prime}}}$.

\end{proof}

{\bf Remark 1}. Theorems \ref{ShKass} and \ref{Comp} are true for
some  other Kerdock-related codes. In particular, by
considerations similar to those in proof of Theorem \ref{ShKass}
one can show that a Kerdock and a shortened Kerdock codes produce
association schemes, which gives an alternative (combinatorial)
proof for the well-known facts from \cite{Del} and \cite{SolTok}.
Analogously to the proof of Theorem \ref{Comp}, one can prove that
the $i$-components of a Kerdock code coincide with the Kerdock
code or equivalently, the minimum distance graph of a punctured
Kerdock code is connected.

\medskip

{\bf Remark 2}. According to Theorem \ref{Comp}, new Kerdock codes
cannot be constructed by means of traditional switchings. For
convenience we set $i=n-1$.
 By the proof
Theorem \ref{Comp} we know that two codewords are in one
$(n-1)$-component of the punctured Kerdock code $K^{*}_n$ if and
only if their puncturings in $(n-1)$th coordinate position have
weights of the same parity. Therefore, the codewords of the
Kerdock code $K$ could be represented as $K^{00}$, $K^{11}$,
$K^{01}$, $K^{10}$, where $K^{ab}=\{x\in K: x_{n-1}=a, x_n=b\}$,
with $K^{00}\cup K^{11}$ corresponding to one $(n-1)$-component of
$K^{*}_n$ and $K^{01}\cup K^{10}$ to the other one. Moreover, the
"odd weight" component is the translation of the "even weight"
one, i.e. there is a codeword $(x'01)$ of $RM(1,m)$ such that
$(K^{01}\cup K^{10})+(x'01)=K^{00}\cup K^{11}$. Now the switching
$K=K^{00}\cup K^{11}\cup ((x'01)+(K^{00}\cup K^{11}))$ to
$K'=K^{00}\cup K^{11}\cup ((x'10)+(K^{00}\cup K^{11}))$ gives an
equivalent code which is obtained from $K$ by permuting $(n-1)$th
and $n$th coordinate positions.

\section{Components of codes dual to BCH codes}

 In the section we fix $n=2^m$, $m$ odd. We investigate the $i$-components of the dual code $C_{1,3}^{\perp}$ of
a primitive cyclic BCH code  $C_{1,3}$ with zeros $\alpha$ and
$\alpha ^3$ with designed distance 5 by $i$-components, of length
$n-1=2^m-1$, $m$ odd, here $\alpha $ is a primitive element of the
Galois field $GF(2^m)$. The code shares many similar properties
with a Kerdock code. We prove that $C_{1,3}^{\perp}$ is an
$i$-component for any coordinate position $i$.

Further we use the following properties of the code
$C^{\perp}_{1,3}$.

(B1) \cite{MWSl} The minimum distance of the code
$C^{\perp}_{1,3}$ is $d=\frac{n-\sqrt{2n}}{2}$. The code
$C^{\perp}_{1,3}$ has the following weight distribution:

\bigskip

\begin{tabular}{|c|c|}
  \hline
  i & The number of codewords of weight i \\
  \hline
  0 & 1 \\
  d & $(n-1)(\frac{n}{4}+\sqrt{\frac{n}{8}})$ \\
  $\frac{n}{2}$ & $(n-1)(\frac{n}{2}+1)$ \\
  n-d & $(n-1)(\frac{n}{4}-\sqrt{\frac{n}{8}})$ \\
  \hline
\end{tabular}

\bigskip

The  fact below follows from Theorem \ref{TMW} and (B1).

 (B2) Fixed weight codewords of $C^{\perp}_{1,3}$ form a
2-design.

The code $C_{1,3}$ is uniformly packed \cite{BZZ}. In \cite{Del},
Theorem 6.10 it was shown that any code that is dual to a linear
uniformly packed code gives an association scheme.

 (B3)\cite{Del} The restriction of the Hamming scheme to $C^{\perp}_{1,3}$ is an association scheme.

\begin{lemma} \label{BCH1}
Let $C$ be the punctured (in any coordinate position) code of the
code $C^{\perp}_{1,3}$. Then any codeword of weight $d$ is at
distance $d-1$ from at least one codeword of weight $d-1$.
\end{lemma}

\begin{proof} Let $C_{d-1}$ be the set of codewords of the punctured code of $C^{\perp}_{1,3}$ of weight $d-1$. Suppose that $x$ is a codeword of weight $d$ such
that $d(x,C_{d-1})>d-1$. Then $d(x,C_{d-1}) \in \{\frac{n}{2}-1,
n-d-1\}$. Since the vectors of $C_{d-1}$ form 1-design which
follows from the property (B2), we can use Lemma \ref{magic} to
count the number $\delta^{\frac{n}{2}-1}$ of the codewords of
$C_{d-1}$ at distance $\frac{n}{2}-1$ from $x$:
$$\delta^{\frac{n}{2}-1} (d-\frac{n}{4}) + (|C_{d-1}| -
\delta^{\frac{n}{2}-1})\frac{3d-n}{2}=\lambda_1\cdot d,$$ where
$|C_{d-1}|=\lambda_1\frac{n-2}{d-1}$.

It is easy to see that
$$\frac{\delta^{\frac{n}{2}-1}}{
|C_{d-1}|} =\frac{2(n^2-2n+8d-3nd-2)}{(n-2)(n-2d)}>1,$$ a
contradiction.
\end{proof}

\begin{lemma} \label{BCH2}
The minimum weight codewords of $C^{\perp}_{1,3}$ span the code.
\end{lemma}
\begin{proof}
The code $C^{\perp}_{1,3}$ is the direct sum of the Hadamard codes
$C^{\perp}_{1}$ and $C^{\perp}_{3}$, both of which consist of
$n-1$ nonzero codewords having weight $n/2$. The number of
codewords of weight $d$ in $C^{\perp}_{1,3}$ is greater then $n$
(see (B1)). Therefore one can find three codewords in codes
$C^{\perp}_{1}$ and $C^{\perp}_{3}$ with distances $d$ or $n/2$
pairwise, e.g. $x,x^{\prime}\in C^{\perp}_{1}$ and $y\in
C^{\perp}_{3}$, such that $d(x,x^{\prime})= n/2$ and
$d(x,y)=d(x^{\prime},y)=d$. Hence, by property (B3), we have that
the intersection number  $\delta_{d,n/2}^d(C^{\perp}_{1,3})$ is
nonzero, i.e. any codeword of weight $n/2$ is at distance $d$ from
at least one codeword of weight $d$ in $C^{\perp}_{1,3}$.

The number of codewords of weight $n-d$ is less than the number of
codewords of weight $d$, therefore any codeword of weight $n-d$ is
at distance $d$ from at least one codeword of weight $n/2$ or $d$.
So, the codewords of weight $d$ generate the code
$C^{\perp}_{1,3}$.
\end{proof}

\begin{theorem}\label{theodualBCH}
A code $C^{\perp}_{1,3}$ of length $n=2^m-1$, $m$ odd, consists of
one $i$-component for any coordinate position $i$.
\end{theorem}

\begin{proof}

By Lemma \ref{BCH1} any codeword of $C^{\perp}_{1,3}$ of weight
$d$ with $0$ in the
 $i$th coordinate position is at distance $d$ from a codeword of
weight $d$ with $1$ in the $i$th coordinate position. By Lemma
\ref{BCH2}, this implies that the set of all codewords of weight
$d$ having $1$ in the $i$th coordinate position generates the code
$C^{\perp}_{1,3}$, i.e. the code $C^{\perp}_{1,3}$ is an
$i$-component for any $i\in \{1,2,\ldots,n-1\}.$

\end{proof}




 Note that the properties (B1)-(B3) and the proof
of Theorem \ref{theodualBCH} are the same for any code that is
dual to a linear uniformly packed code with the same parameters as
the BCH code. In particular, the cyclic code $C_{1,2^j
+1}^{\perp}$, $(j,m)=1$ corresponding to the Gold function,
 $n-1=2^m-1$, $m$ odd as well as the duals of other linear codes obtained from almost bent functions (AB-functions) are uniformly packed \cite{CCZ}
 and therefore each of them is an $i$-component for any $i$.

\begin{corollary}\label{theodualBCH}
The dual of a linear uniformly packed code with parameters of BCH
code $C_{1,3}$ of length $n-1=2^m$, $m$-odd
 is an $i$-component for any coordinate position $i$.

\end{corollary}

{\bf Conclusion.} We considered duals of two such well-known
classes of uniformly packed codes as Preparata and 2-error
correcting BCH code. The dual codes have large minimum distance,
few nonzero weights and are related to designs and association
schemes. We proved that $i$-components of these codes are maximum.
It would be natural to study the structure of $i$-components of
Preparata codes that are formal duals of Kerdock codes. For $n=15$
these classes meet in the self-dual Nordstrom-Robinson code that
has two $i$-components for any coordinate position $i$. With the
help of a computer, we showed that $C_{1,3}^{\perp}$ of length
$2^m-1$ is an $i$-component for any $i$ for even $m$ also for
$m=6, 8, 10$ and the BCH code $C_{1,3}$ consists of two
$i$-components for any coordinate position $i$ for any $m$: $5\leq
m\leq 8$. Another challenging problem is finding $i$-components of
the BCH codes $C_{1,3}$ for any $m$ and their duals for even $m$.

\end{document}